\documentclass{amsart}

\newtheorem{theorem}{Theorem}[section]

\newtheorem{proposition}[theorem]{Proposition}

\usepackage{mathrsfs}

\theoremstyle{definition}

\newtheorem{example}[theorem]{Example}

\theoremstyle{remark}
\newtheorem{remark}[theorem]{Remark}

\numberwithin{equation}{section}

\newcommand{\I}{{\mathds {1}}}
\newcommand{\Rdb}{{\mathbb R}}
\newcommand{\Cdb}{{\mathbb C}}
\newcommand{\Tdb}{{\mathbb T}}
\newcommand{\Ddb}{{\mathbb D}}

\begin{document}

\title{The Hoffman-Rossi theorem for operator algebras}

\author{David P. Blecher}
\address{Department of Mathematics \\ University of Houston \\  Houston,  TX
77204-3008, USA}
\email[David Blecher]{dblecher@math.uh.edu}
\author{Luis C. Flores}
\address{Department of Mathematics \\ University of Missouri \\ Columbia, 
MO 765212, USA}
\email[Luis Flores]{lcfvd6@mail.missouri.edu}

\author{Beate G. Zimmer}
\address{Department of Mathematics and Statistics \\ Texas A $\&$ M University--Corpus Christi \\ 
Corpus Christi, TX 78412-5825, USA}
\email[Beate Zimmer]{beate.zimmer@tamucc.edu}

 \date{\today}
\thanks{DB is supported by a Simons Foundation Collaboration Grant.}
\begin{abstract}   We study possible noncommutative (operator algebra)
 variants of the classical Hoffman-Rossi theorem from the theory of
function algebras.   In particular we give a condition on the range of a contractive weak* continuous homomorphism
defined on an operator algebra $A$, 
which is necessary and sufficient (in the setting we explain) for a positive weak* continuous extension to any von Neumann algebra containing $A$.
\end{abstract}

\maketitle

\section{Introduction}

The Hoffman-Rossi theorem that we are interested in here
is a remarkable result from the theory of
function algebras \cite[Theorem 3.2]{HR}.  It states 
 that if $A$ is a weak* closed unital subalgebra of $M = L^\infty(\mu)$, for a probability measure $\mu$, 
and if $\varphi$ is a weak* continuous character (i.e.\ nontrivial complex-valued homomorphism)
on $A$, then  $\varphi$ has a weak* continuous positive
linear extension to $M$.   Since unital linear functionals on $M$ (or on any $C^*$-algebra) are states (that is,
contractive and unital linear functionals) 
if and only if they are positive,
this is saying that weak* continuous characters on $A$
have weak* continuous Hahn-Banach extensions  to $L^\infty(\mu)$. 
Or, in the language of von Neumann algebras (see e.g.\ p.\ 245, 248--249 in \cite{Bla}),
weak* continuous characters
on $A$ have {\em normal} state extensions to $L^\infty(\mu)$.   In the present paper we study
 possible noncommutative (operator algebra) variants of this result.     An operator algebra
is a unital algebra of operators on a Hilbert space, or more abstractly a Banach algebra isometrically isomorphic to  such an algebra of Hilbert space operators.   (Sometimes one wants to consider 
an operator space structure on an operator algebra, and replace the word `isometrically' by 
`completely isometrically' in the last sentence (see e.g.\
\cite{BLM} for definitions), but this will not be important in the present paper.)  Our main result is  a condition on the range of a contractive (or completely contractive) weak* continuous homomorphism $\Phi$ 
defined on a unital operator algebra $A$, 
which is necessary and sufficient  in the setting explained below, for a positive weak* continuous extension
of  $\Phi$  to any von Neumann algebra containing $A$
as a weak* closed subalgebra.

  \section{A von Neumann algebra
valued Hoffman-Rossi theorem}
 
There are several natural ways to try to generalize the Hoffman-Rossi theorem to the 
operator algebra setting.   
First however we note that the setting of algebras and algebra homomorphisms
is crucial.
Even in the classical setting, one cannot hope that 
weak* continuous states on operator systems have weak* continuous Hahn-Banach extensions
(that is, normal state extensions).  The following, which we will use  later,
  is a convincing illustration
of this.

\begin{example} \label{ex1}  Consider the state $\varphi_1$ of evaluation at $1$ on the set 
${\mathcal S}$ of polynomials of degree $\leq 1$
on $[0,1]$, viewed as a subspace of $L^\infty([0,1])$. This
is weak* continuous (since ${\mathcal S}$ is finite dimensional).  
Any normal state extension to $L^\infty([0,1])$ of $\varphi_1$
is integration against a positive
$g \in L^1([0,1])$ with $\int_{[0,1]} \, g \, dt= 1$.   Applying this state
to the monomial $t$ gives $\int_{[0,1]} \, t g(t) \, dt = \varphi_1(t) = 1$.
Hence $\int_{[0,1]} \, (1-t) \, g(t) \, dt= 0$, forcing the contradiction
$g = 0$ a.e..   \end{example}

 With a little more work one can find a weak* continuous state on a 
unital weak* closed subalgebra of a von Neumann algebra $M$, which has no
normal state extension to $M$ (for example  the algebra in the proof of Proposition \ref{exa} below).

In the  noncommutative (operator algebra) setting we suppose that we have 
a weak* closed unital subalgebra $A$ of 
a von Neumann algebra $M$, and a weak* continuous unital 
contractive homomorphism  $\Phi : A \to D$, for a von Neumann algebra $D$.  
By the first (resp.\ second) `unital' here we mean that $1 \in A$, where $1$ is the identity of $M$
(resp.\ $\Phi(1) = 1$).
The question we are interested in is when   does $\Phi$ have a weak* continuous 
contractive (or equivalently, positive) 
$D$-valued linear extension to $M$? 
Sometimes we will add the adjective `completely', for example
consider UCP (unital completely positive) extensions--see e.g.\
\cite[Chapter 1]{BLM} for notation.    To obtain such a
Hoffman-Rossi theorem  one has to have restrictions on the algebras
or on $\Phi$ or its range of $\Phi$, as we will see below.  It is not true
in general, for example, when $D = B(H)$ for a Hilbert space $H$
(unless $A$ is also selfadjoint, in which case it may be proved using
as one ingredient the well
known extendibility of normal representations of von Neumann algebras).   

In the case of the original 
Hoffman-Rossi theorem we may identify the range of the character with $D = \Cdb 1_A = \Cdb 1_M$, 
in which case the homomorphism we are extending is an idempotent map 
on $A$, and is a $D$-bimodule map.
Thus we will usually restrict our attention 
to the setting  of a weak* continuous unital 
contractive  or completely contractive homomorphism $\Phi : A \to D \subset A$ 
which is a $D$-bimodule map (or equivalently, is the identity map on $D$).  Such maps $\Phi$ are called
$D$-{\em characters} in a sequel  paper \cite{BLvv}, where their theory is developed much more extensively.
We remark that by
4.2.9 in \cite{BLM}, any unital completely contractive projection of a unital operator algebra onto a subalgebra
$D$ is a  $D$-bimodule map. 
We then 
ask for a positive normal extension from $M$ into $D$.  Note that the latter implies that $D$ is a von Neumann algebra, and $\Psi$ is
a conditional expectation onto $D$ (see p.\ 132--133 in \cite{Bla} for the main facts about conditional expectations and their
relation to bimodule maps and projections maps of norm $1$).
However $D$ being a von Neumann algebra is not enough for a weak* continuous positive extension from $M$  into $D$, even if $D$ is also injective
and commutative.   
For example,  Takesaki \cite{Tak} showed that there need not exist a positive weak* continuous extension from $M$ onto such $D$ if $M$ is not `finite'.  

The correct Hoffman-Rossi theorem in the setting described in the last paragraph that works 
for all von Neumann algebras $M$ 
requires  $D$ to be finite dimensional, or more generally, a (purely) atomic von Neumann 
algebra (see p.\ 354 in \cite{Bla}). 
This is Theorem \ref{HRvn}, and the result after it (Proposition \ref{niff}) 
shows the necessity of the atomic hypothesis.

First however we give a quick proof of  the case  of our main result in the scalar valued case.  This particular proof we noticed some
time after our paper was submitted (our original proof is now Remark  \ref{nr}).   It follows
a method which seems to be attributed by Hoffman and Rossi to
 Heinz K\"onig and Don Sarason below Theorem 3.3 in \cite{HR}.  These
authors of course addressed the commutative case.  
A variant of this argument may be found in
a (much later) method of Cassier to prove a similar result for characters on
singly generated commutative dual  operator algebras
\cite{Cas}.  These ideas are somewhat buried in \cite{HR} and \cite{Cas}.  Indeed  
the proof of the Hoffman-Rossi result (stated
in our very first paragraph) which 
is officially presented in \cite{HR} and which 
found its way into the function algebra texts (such as \cite{Gam})
is several orders of magnitude
more complicated. 
We combine and modify these  ideas  hidden in \cite{HR,Cas}, 
so as to include both the noncommutative case and (a generalization of)
 the original Hoffman-Rossi theorem.  Of course the statement in the 
commutative case is formally 
contained in the noncommutative result.
We have chosen to write it as follows to make the importance
 of K\"onig and Sarason's contribution clearer.   
In the  description of their contribution below Theorem 3.3 in \cite{HR}, $d \mu = r \, dm$ in our notation below; 
 $m$ is a positive measure such that $L^1(m)^* = L^\infty(m)$.

\begin{theorem} \label{HRks} {\rm (Hoffman-Rossi, the scalar valued case)} \  
Let $M = B(H)$ for a Hilbert space $H$
or let $M= L^\infty(m)$ be a commutative von Neumann algebra.
Suppose that $\varphi : A \to \Cdb$ is a weak* continuous character on a weak*
closed subalgebra $A$ of $M$ containing the unit of $M$.  Then
$\varphi$ has a normal state extension to $M$.
\end{theorem} \begin{proof}  To combine the classical and
noncommutative case we write $L^\infty$ for $B(H)$ or $L^\infty(m)$,
$L^1$ for the predual of $L^\infty$, and tr for the trace on $B(H)$ or for the
integral on $L^\infty(m)$ (classical case).
We also write $L^2$ for the Hilbert space with inner product tr$(b^* a)$
(this will be the Hilbert-Schmidt operators in the $B(H)$ case or $L^2(m)$
in the commutative case).  By Banach space
duality there exists $r \in L^1$ such that tr$(xr) = \varphi(x)$ for $x \in A$.
Write $r = ab$ for $a, b \in L^2$.  Let $J = {\rm Ker} \, \varphi$, and
let $E$ (resp.\ $F$) be the
closure in the $L^2$  norm of $Aa$ (resp.\ $Ja$).
 For $f \in J$ we have
$$\| a - f a \|_2^2 = {\rm tr} \, (a^* |1-f|^2 a) = {\rm tr} \, (|(1-f)a|^2) 
\geq \frac{1}{C^2},$$ for a constant $C > 0$, since
 $$1 = \varphi(1-f) = {\rm tr} \, ((1-f)r) =
{\rm tr} \, (b (1-f)a) \leq C \, {\rm tr} \, (|(1-f)a|^2)^{\frac{1}{2}} .$$ 
  It follows that $a \notin F$, so $E \neq F$.  Choosing a
unit vector $h \in E \ominus F$, we have
 $${\rm tr} \, (f |h^*|^2) = \langle fh , h \rangle = 0 ,
\qquad f \in J,$$ since $fh \in J E \subset F$.
Since $A = J + \Cdb 1$ it follows that tr $(\cdot |h^*|^2)$ is a normal state on $M$
extending $\varphi$.
\end{proof}


\begin{remark} \label{nr}  An alternative quick 
(and more noncommutative) proof which also relies on the  $A = J
+ {\mathbb C} 1$ relation:   Suppose that $M$ acts on  a Hilbert space $K$, so that $A \subset B(K)$. Consider the amplification $\pi(x) = x^{(\infty)}$
of the identity
representation  of $M$,
acting on the countably infinite direct sum $K^{(\infty)}$ of copies of $K$.
Then $\pi(J)$ is reflexive by e.g.\ \cite[Appendix A.1.5]{BLM}, and by definition
of reflexive there, since
$I \notin \pi(J)$, there is a vector $\xi \in K^{(\infty)}$ with
$\xi \notin [\pi(J) \xi]$.   Hence $[\pi(A) \xi] \ominus [\pi(J) \xi] \neq (0)$.
Choose $\eta \in [\pi(A) \xi] \ominus [\pi(J) \xi]$ of norm $1$.
Since $\pi(J) \eta \in \pi(J) [\pi(A) \xi] \subset [\pi(J) \xi]$ it follows that
$\langle \pi(\cdot) \eta , \eta \rangle$ 
is a normal state on $M$ annihilating $J$.   Hence its restriction to $A$
is a multiple of $\Phi$, indeed equals $\varphi$ since both are states. 
\end{remark}
 
\begin{theorem} \label{HRvn} Consider the inclusions $D \subset A \subset M,$ where $M$ is a von Neumann algebra, $A$ is a
weak* closed subalgebra of $M$, and $D$ is an atomic von Neumann subalgebra containing the unit of $M$. If $\Phi : A \to D$ 
is a unital weak* continuous homomorphism which is also a $D$-bimodule map, then $\Phi$ extends to a normal
conditional expectation $\Psi : M \to D$.
\end{theorem}

\begin{proof}
The case that $D = \Cdb 1$ follows immediately from Theorem \ref{HRks}.
 
Next suppose that $D$ is a type I factor, thus isomorphic to $B(l^2(I))$ for an index set $I$.  We may suppose that
$M = B(K)$, and that $D$ is the range of a normal unital
$*$-homomorphism $\pi : B(l^2(I)) \to B(K)$.  For $i \in I$ let $p_i = \pi(E_{ii})$, where $\{ E_{ik} \}$ are the matrix
units in $B(l^2(I))$,
and set $L = p_j K$ for a fixed $j \in I$.  By a
matrix unit argument we may suppose that (unitarily)
$K = L^{(I)} = L \otimes l^2(I)$, the Hilbert space sum of $I$ copies of $L$,
that $M = B(L) \bar{\otimes} B(l^2(I))$,
and $D = I_L \otimes B(l^2(I))$ and
$p_i = I_L \otimes E_{ii} \in D \subset A$.
Set $A_j =  p_j A p_j$ and 
$$B = \{ T \in B(L) : p_j (T \otimes I)  p_j \in A_j  \subset A \} .$$ Then by a routine matrix unit argument we have 
 $B  \cong B \otimes E_{jj} =  A_j$ and
$A = B \bar{\otimes} B(l^2(I))$ via the unitary above.   Here $B \bar{\otimes} B(l^2(I))$ is the {\em normal spatial tensor product} of e.g.\ 2.7.5 (2) in \cite{BLM}, namely the 
weak* closure of $B \otimes B(l^2(I))$ in the von Neumann algebra $B(L) \bar{\otimes} B(l^2(I)) \cong B(K)$.  
That $A = B \bar{\otimes} B(l^2(I))$  follows by a  matrix unit argument (exploiting the usual properties of
matrix units like $E_{ik} = E_{ij} E_{jk} = E_{ij} E_{jj}  E_{jk}$) and 
weak* density.     One
also uses the fact that $A$ contains, and is a submodule over, $D = I_L \otimes B(l^2(I))$, that $A_j =  B   \otimes E_{jj}$, and that the span of the $p_i A p_k$ is weak* dense in $A$.

For $i, k \in I$, $\Phi (p_i a p_k) = p_i \Phi (a) p_k \in \Cdb E_{ik}$.  Hence
$\Phi (b \otimes E_{ik})$, viewed as a matrix in
$B(l^2(I))$, is zero except perhaps for its $i$-$k$ entry,
for any $b \in B$.
Let $\varphi = \pi_{jj} \circ \Phi \circ \epsilon_{j}$, where
$\epsilon_j(b) = b \otimes E_{jj}$ and $\pi_{jj}$ is the state on $D$ that evaluates the $j$-$j$ entry.
Then $\varphi$ is a character of $B$.
Also, $\Phi = \varphi \otimes I$, indeed for $b \in B$ and $i, k \in I$ we have
$$\Phi (b \otimes E_{ik}) = E_{ij} \Phi (b \otimes E_{jj}) E_{jk} =  E_{ij} \varphi (b) E_{jk} =   ( \varphi \otimes I)(b \otimes E_{ik}).$$ 
 first paragraph, $\varphi$ extends to a normal state
 $\sigma$ on $B(L)$, so that $\Psi = \sigma \otimes I$ is a
normal UCP map extension of $\Phi$ to $M = B(L) \bar{\otimes} B(l^2(I))$.

Finally, suppose that $D$ is atomic, so $D \cong \oplus_{i \in I} \, B(H_i)$
for Hilbert spaces $H_i$.   Let $p_i$ be the central projection
in $D$ corresponding to the identity in $B(H_i)$.  
We have 
 that
 $\Phi (p_i a p_j) = p_i \Phi (a) p_j = 0$ for $i \neq j$ (since $p_i$ is central in $D$).
Thus $\Phi = \Phi \circ \Delta_{|M}$ where $\Delta$ is
the UCP map on $M$ defined by $\Delta(x) = \sum_i \, p_i x p_i$.
Let $\Phi_i = \Phi_{|p_i A p_i}$.   By the  last paragraph  $\Phi_i$
extends to a normal UCP map $\sigma_i : p_i M p_i \to p_i D p_i$.
We obtain a normal UCP map extension $\Psi = (\oplus_i \, \sigma_i) \circ \Delta$
of $\Phi$ to $M$.   By Tomiyama's 
well known theorem (see p.\ 132--133 in \cite{Bla})   
 on projections of norm $1$,
$\Psi$ is necessarily a $D$-bimodule map and conditional expectation. 
 \end{proof}

\begin{proposition} \label{niff}  Suppose that $A$ is a unital 
weak* closed subalgebra of $M = B(H)$, and suppose that $D$ is a weak*
 closed unital  selfadjoint subalgebra of $A$. 
Suppose that $\Phi : A \to D$ is a weak*
 continuous unital homomorphism on $A$ that
is a $D$-bimodule map.    Suppose that  
  there exists a  normal positive map $\Psi : M \to D$ extending $\Phi$.   Then $D$ is an atomic von Neumann algebra. 
\end{proposition}

\begin{proof}  
Note that such a map $\Psi$ is necessarily a normal norm $1$ projection 
onto $D$.   However a von Neumann algebra which is the range of such a projection on $B(H)$ is
atomic (\cite[Theorem IV.2.2.2]{Bla}).  \end{proof}   

The following example shows the importance of the selfadjointness of $D$
in finding {\em any} positive $B(H)$-valued extension of a homomorphism.

\begin{proposition} \label{exa}    There exists  inclusions $D \subset A \subset M,$ where $M$ is a finite von Neumann algebra, and  $A$ and $D$ are 
commutative finite dimensional  weak*
 closed subalgebras of $M$  containing the unit of $M$, with the following properties.  \begin{itemize} 
\item [(1)] 
The subalgebra $D$ may be identified (completely isometrically)
with a unital subalgebra $D$ of the $2 \times 2$ 
matrices $M_2$.  \item [(2)] 
There exists a  weak* continuous
completely contractive unital homomorphism $\Phi : A \to D \subset M_2$,
such that
 $\Phi$ has no contractive or positive
weak* continuous linear extension from $M$ to $M_2$ or to $D$.
\item [(3)] If $\Phi$ is viewed as a map
$A \to D \subset A$ then $\Phi$ is an idempotent $D$-bimodule map.
\end{itemize} 
\end{proposition}

\begin{proof}
Let $\varphi_1$ be the weak* continuous state on ${\mathcal S}
\subset L^\infty([0,1])$ in Example  \ref{ex1}.
Let $M = M_2(L^\infty([0,1]))$, the $W^*$-algebra of $2 \times 2$
matrices with entries in $L^\infty([0,1])$.
Let $A$ be the subalgebra
of $M_2(L^\infty([0,1]))$ consisting of upper triangular matrices
with scalars (constant functions) as the main diagonal entries and elements from
${\mathcal S}$ in the $1$-$2$ entry.  This $A$ is four dimensional.
It is also  weak* closed in $M$ since any finite dimensional subspace is
closed in any linear topology.  This is related to \cite[Lemma 2.7.7]{BLM}.
Define
$\Phi : A \to M_2$ to be the map that applies $\varphi_1$ in the $1$-$2$ entry,
and leaves other entries `unchanged'.  That is $\Phi$ is `evaluation' at $1$.
This is easily seen
to be a weak* continuous unital homomorphism.
 Also, $\Phi$ is (completely) contractive by \cite[Proposition 2.2.11]{BLM}.
Let $D$ be the range of $\Phi$, the upper triangular
$2 \times 2$
matrices.   We may also view $D \subset A$ by identifying an upper triangular
matrix with the same matrix in $A$, but with $1$-$2$ entry multiplied by the
monomial $t$.  Then $\Phi$ viewed as a map
$A \to D \subset A$ is an idempotent $D$-bimodule map.
  Suppose that $R : M \to M_2$ was
a weak* continuous contractive extension of $\Phi$.  Then  
$R$ is positive (since it is well known that contractive unital maps
on $C^*$-algebras are positive).   The restriction of $R$
to matrices that are only nonzero in their $1$-$2$ entry, 
followed by the
projection onto the $1$-$2$ entry, defines a weak* continuous contractive
functional $\psi$ on $L^\infty([0,1])$.    
 Finally, it is clear that
$\psi$ extends $\varphi_1$, contradicting the first paragraph of this section.  
This contradiction shows that our extension $R$ cannot be
positive or contractive.
\end{proof}

One may adjust $A$ in the proof above
to be three dimensional by taking the
main diagonal entries of matrices in $A$ to be equal.    We also remark that if one insists on bimodule map 
extensions then one may get counterexamples with $M$ finite dimensional 
(see e.g.\ \cite[Example 3.5]{Smi}).

We also remark that very strong forms
of the noncommutative Hoffman-Rossi theorem
hold for Arveson's maximal subdiagonal subalgebras  of 
$\sigma$-finite von Neumann algebras \cite{BLueda}
(or more generally 
for algebras having some  of the Gleason-Whitney properties GW1, GW2, GW
from the start of Section 4 in \cite{BL-FMR}, or their variants for states, 
and studied further in e.g.\ \cite[Section 5]{BLueda} in the $\sigma$-finite 
case).    See Section 3 of \cite{BLvv} for more on this.

\subsection*{Acknowledgment}    Noncommutative Hoffman-Rossi theorems were a project suggested (and
guided in its more technical parts, e.g.\ things involving von Neumann algebras) by the first author
for the second author's  M. S.\  thesis  \cite{LFthesis} supervised by the third author.  
The
present paper contains several advances made subsequent to that reference, including the main result. 
We also thank the referee for several comments.

\end{document}